\documentclass[11pt, oneside]{article}   	
\usepackage{geometry}                		
\usepackage{graphicx}				
\usepackage{amsmath,amssymb,amsthm,enumerate,comment,color, bbm}
\usepackage{hyperref}
\usepackage{mypackage}
\usepackage[T1]{fontenc}
\usepackage{enumitem}

\renewcommand{\P}{\mathcal{P}}
\newcommand{\W}{\mathcal{W}}
\theoremstyle{plain}
  \newtheorem{theorem}{Theorem}[section]
  
  \newtheorem{corollary}[theorem]{Corollary}
  \newtheorem{lemma}[theorem]{Lemma}

\theoremstyle{definition}
  
  \newtheorem{non-example}[theorem]{Non-Example}
  \newtheorem{definition}[theorem]{Definition}

  \newtheorem{remark}[theorem]{Remark}

\DeclareMathOperator*{\argmin}{arg\,min}
\let\div\relax
\DeclareMathOperator{\div}{div}
\DeclareMathOperator{\proj}{proj}
\renewcommand{\F}{\mathcal{F}}
\title{A  Linear Equation on the Set of Probability Vectors on Graphs}
\author{Miho Kasai}
\begin{document}
\maketitle

\begin{abstract}
\noindent
    In this paper we investigate solutions to a linear Hamilton-Jacobi equations in the Wasserstein space of probability vectors on a finite simply connected graph. We prove that there exists a solution under the assumption that the initial value function $u_0:\P(G)\to\R$ is Fr\'echet continuously differentiable.
\end{abstract}
\section{Introduction}

The goal of this paper is to understand the search of Nash equilibria in game theory with finitely many states $\{t_1, \dots, t_n\}$ which we will denote with $\{1, \dots, n\}$ with infinitely many players. 

Define the following value function in the continuum setting 

\[
v(t,x):=\min_{\gamma: \gamma(t)=x}\bigg{\{}\int_{t}^{T}\mathcal{L}(\gamma,\dot{\gamma})ds+\Phi(\gamma(T))\bigg{\}}
\]

Here $x$ belongs to a Hilbert space or to a quotient space of a Hilbert space.

It is well-known \cite{PDE} in the theory of calculus of variations that $v(t,x)$ satisfies the Hamilton-Jacobi equations
\begin{align*}
\partial_t v(t,x)+H(x,\partial_x v)&=0 \\
v(T, \cdot)&=\Phi(x)
\end{align*}

Much work has been done to understand the Hamilton-Jacobi equations in various space of measures. For example, the study of viscosity solutions in the Wasserstein space $\P_2(\R^d)$ is presented in \cite{real-W}, and in \cite{torus} for $\P(\mathbb{T}^d)$ where $\mathbb{T}^d$ is the $d$-dimensional torus. In this paper, we will study the Hamilton-Jacobi equations on the set of probability vectors on graphs, which we will denote by $\P(G)$. Define $G=(V, E, w)$ to be a simply connected undirected graph with the set of vertices $V=\{1, \dots, n\}$ and the set of edges $E\subset V^2$. Further, the weight $\omega=(\omega_{ij})$ is a $n$ by $n$ symmetric matrix with $\omega_{ij}>0$ if $(i,j)\in E$, and $\omega_{ij}=0$ otherwise.
 Note that 
\[ \P(G)=\{\rho\in\R^n: \sum \rho_i=1, \rho_i\geq 0\}\]. \\

We use $g:[0,1]\times [0,1]\to [0, \infty)$ as a metric tensor that satisfies the following properties.\\
\begin{enumerate}
\item $g\in C([0,1]^2)\cap C^{\infty}((0,1)^2)$ 
\item $g(s, t)=g(t, s)$ 
\item $g((1-\lambda)a+\lambda b)\geq (1-\lambda)g(a)+\lambda(b) \quad \forall \lambda\in(0,1)$ 
\item $g(\lambda s,\lambda t)=\lambda g(s, t) \quad \forall \lambda>0$ 
\item $\min\{s,t\}\leq g(s,t)\leq \max\{s,t\}$
\end{enumerate}
\noindent
\textbf{Examples}
\begin{itemize}
    \item[(i)] $g(s,t)=\cfrac{s+t}{2}$
    \item[(ii)] $g(s,t)=\begin{cases}\cfrac{2}{\frac{1}{s}+\frac{1}{t}} & st\neq 0\\
    0 & st=0
    \end{cases}$
    \item[(iii)] $g(s,t)=\begin{cases}\cfrac{s-t}{\log s-\log t} & st\neq 0, s\neq t\\
    t & st\neq 0, s=t \\
    0 & st=0
    \end{cases}$
\end{itemize}

For $v\in\mathbb{S}^{n\times n}$ and $\rho\in \P(G)$, define graph-divergence
\[\langle \text{div}_{\rho}(v) \rangle_i=\sum_{j\in N(i)} \sqrt{\omega_{ij}}g_{ij}(\rho)v_{ji}\]
where $N(i)$ is the set of vertices connected to $i$, and $g_{ij}(\rho):=g(\rho_i, \rho_j)$. 

Further, for $v, \bar{v}\in\mathbb{S}^{n\times n}$, define graph-inner product and norm
\[(v, \bar{v})_\rho=\frac{1}{2}\sum_{(i,j)\in E}g_{ij}(\rho)v_{ij}\bar{v}_{ij}\]
\[\Vert v\Vert_{\rho}^2=(v,v)_\rho\]

Formulating the problem in terms of PDEs, we are given with the initial value function
\begin{equation}
u_0: P(G)\to\mathbb{R}
\end{equation}
and a running cost 
\begin{equation}
\mathcal{L}: \P(G)\times \mathbb{S}^{n\times n}\to\mathbb{R}
\end{equation}
where $\mathbb{S}^{n\times n}$ is the set of $n\times n$ skew-symmetric matrices and a noise intensity $\epsilon\geq 0$.

We want to solve the Hamilton-Jacobi equations
\begin{equation} \label{HJE'}
\begin{cases}
\partial u(t, \mu) + \mathcal{H}(\mu, \nabla_{\mathcal{W}}u(t,\mu))=\epsilon \Delta_{ind}u(t,\mu) \\
u(0, \mu)=u_0
\end{cases}
\end{equation}
Here \[\mathcal{H}(\mu, p)=\sup_{v\in \mathbb{S}^{n\times n}} (v, p)_\mu-\mathcal{L}(\mu, v) \text{ where } p\in\mathbb{S}^{n\times n}\] and the individual noise operator which is first defined in \cite{graph} is
\[\Delta_{ind}u(t, \mu)=\mathcal{O}_\mu(\nabla_\mathcal{W} u(t, \mu))\] where 
\[\mathcal{O}_\mu(p)=(\text{div}_\mu(p), \log\mu)_\mu\]

The existence of a solution to \eqref{HJE'} is shown in \cite{graph} when there exists $\kappa\in (1, \infty)$ such that 
\begin{equation} \label{eq:4}
\forall \epsilon>0 \quad \exists \theta_\epsilon \quad s.t. \quad \theta_\epsilon \Vert p\Vert_{\mu}^{\kappa}\leq \mathcal{H}(\mu, p)\quad \forall p\in\mathbb{S}^{n\times n}
\end{equation}

Note that \eqref{eq:4} is not satisfied when $\mathcal{H}(\mu, p)=0$. The objective of this paper is to study this case. For simplicity, we may assume that $\epsilon=1$, since if $u^{\epsilon}(t,\mu)=u(\epsilon t, \mu)$, then
\[ \partial_t u^{\epsilon}(t, \mu)=\epsilon \partial_t u\bigg(\frac{t}{\epsilon}, \mu\bigg)=\epsilon \Delta_{ind}u\bigg(\frac{t}{\epsilon}, \mu\bigg)\].

Thus we are concerned with solving 
\begin{equation} \label{HJE}
\begin{cases}
\partial u(t, \mu) = \Delta_{ind}u(t,\mu) \\
u(0, \mu)=u_0
\end{cases}
\end{equation}

\section{Preliminaries}\label{section:1}

\begin{lemma}

For $\phi\in\R^n$, $\rho\in \P(G), v\in\mathbb{S}^{n\times n}$, we have the following integration by parts formula
\begin{equation}
    (\phi, \div_\rho(v))=-(\nabla_G\phi,v)_\rho
\end{equation}
\end{lemma}
\begin{proof}
\begin{align*}
(\nabla_G\phi,v)_\rho&=\frac{1}{2}\sum_{(i,j)\in E}(\nabla_G\phi)_{ij}v_{ij}g_{ij}(\rho) \\
&=\frac{1}{2}\sum_{(i,j)\in E}\sqrt{\omega_{ij}}(\phi_i-\phi_j)v_{ij}g_{ij}(\rho) \\
(\phi, \div_\rho(v))&=\sum_{i=1}^{n}\phi_i(\div_{\rho}(v))_i\\
&=\sum_{i=1}^{n}\phi_i\sum_{j\in N(i)}\sqrt{\omega_{ij}}v_{ji}g_{ij}(\rho) \\
&=\frac{1}{2}\sum_{(i,j)\in E}(\phi_i-\phi_j)\sqrt{\omega_{ij}}v_{ji}g_{ij}(\rho) \\
&=-(\nabla_G\phi,v)_\rho
\end{align*}  
\end{proof}
\noindent
\begin{definition}\textbf{(Velocity)}

Let $\sigma\in C([0,T],\P(G))$ and $v:(0,T)\to \mathbb{S}^{n\times n}$. We say that $v$ is a velocity for $\sigma$ if \begin{equation}
    \dot{\sigma}+\div_\sigma(v)=0
\end{equation}
This is analogous to the definition of velocity field in fluid mechanics.
\end{definition}

\begin{lemma}
Assume that $v\in C([0,T],\mathbb{S}^{n\times n})$, $\sigma\in (C[0,T],[0,1]^n)$ and $\dot{\sigma}+\div_\sigma(v)=0$. 
\begin{equation}
\text{If }\sum_{i=1}^n\sigma_i(0)=1 \text{, then } \sum_{i=1}^{n}\sigma_i(t)=1 \quad \forall t\in[0,T]
\end{equation}
In other words, if $\sigma\in \P(G)$ at $0$, it remains in $\P(G)$ for all $t$ provided that we know $\sigma_i\geq 0$ for all $i$.
\end{lemma}
\begin{proof}
    \begin{align*}
        \frac{d}{dt}\bigg(\sum_{i=1}^n\sigma_i(t)\bigg)&=\sum_{i=1}^{n}\sigma_i(t)
        =\sum_{i=1}^{n}-\langle \div_{\sigma}(v)\rangle_i
        =-\sum_{i=1}^{n}\sum_{j\in N(i)}\sqrt{\omega_{ij}}g_{ij}(\rho)v_{ij} 
    \end{align*}
    Thus \[\frac{d}{dt}\bigg(\sum_{i=1}^n\sigma_i(t)\bigg)=-\sum_{(i,j)\in E}\sqrt{\omega_{ij}}g_ij(\rho)(v_{ij}-v_{ji})
        =0\]

\end{proof}

\noindent
\begin{definition}\textbf{(Poincar\'e function)}

Given $\rho\in\P(G)$, we define 
\begin{equation}
    \gamma_{\text{Poincar\'e}}(\rho)=\inf_{\beta\in\R^n}\bigg\{ \frac{1}{2}\sum_{(i,j)\in E}g_{ij}(\rho)(\beta_i-\beta_j)^2\omega_{ij}:\sum_{i=1}^n\beta_i=0, \sum_{i=1}^n\beta_i^2=1\bigg\}
\end{equation}
\end{definition}

\begin{definition}
    We further define $P_\epsilon(G)=(\epsilon,1]^n\cap P(G)=\{\rho\in\P(G):\rho_i>\epsilon \quad \forall i\}$ if $\epsilon\in[0,1)$
\end{definition}

\begin{lemma}
    If $\rho\in\P_0(G)$ then $\gamma_{\text{Poincar\'e}}(\rho)>0$.
\end{lemma}
    
\begin{proof}
    Note that since $\P(G)$ is compact and $\beta\mapsto \Vert\nabla_G\beta\Vert_\rho^2$ is continuous, $\gamma_{\text{Poincar\'e}}(\rho)$ is the minumum, and $G$ is connected.

    Assume $\gamma_{\text{Poincar\'e}}(\rho)=0$. Then there exists$\beta\in\R^n$ such that $\sum \beta_i=0,\sum\beta_i^2=1$ that satisfies \[\cfrac{1}{2}\sum_{(i,j)\in E}g(\rho_i,\rho_j)(\beta_i-\beta_j)^2\omega_{ij}=0\].

    But $g(\rho_i,\rho_j)(\beta_i-\beta_j)^2\omega_{ij}\geq0$, so $g(\rho_i,\rho_j)(\beta_i-\beta_j)^2\omega_{ij}=0\quad\forall (i,j)\in E$. \\
    Since $g(\rho_i,\rho_j)>0$, $(\beta_i-\beta_j)^2\omega_{ij}\geq0\quad\forall (i,j)\in E$. Fix $i,j\in (1, \dots, n)$. Then there exists $i=l_0, l_1, \dots, l_m=j$ such that $w_{l_{k}l_{k+1}}>0$ for $k\in \{0, \dots, m-1\}$. \\Thus $(\beta_{l_{k}}-\beta_{l_{k+1}})^2w_{l_{k}l_{k+1}}=0$, so $\beta_{l_{k}}=\beta_{l_{k+1}}$ for $k\in \{0, \dots, m-1\}$. Hence $\beta_i=\beta_j$. However this contradicts that $\sum \beta_i=0$ and $\sum \beta_i^2=1$.
\end{proof}

\begin{theorem}\label{thm:Poincare}
    Let $\phi\in\R^n$. Then there exists $\tilde{\phi}\in\R^n$ such that the following holds.
    \begin{itemize}
        \item[(i)] $\nabla_G\tilde{\phi}=\nabla_G\phi$\\
        \item[(ii)] $\Vert\tilde{\phi}\Vert_{l_2}<\Vert\phi\Vert_{l_2}$ unless $\sum\phi_i=0$ \\
        \item[(iii)]$(\tilde{\phi},\mathbbm{1})_{l_2}=0$ \\
        \item[(iv)]$\Vert\tilde{\phi}\|_{l_2}^2\gamma_{\text{Poincar\'e}}(\rho)\leq\Vert\nabla_G\phi\Vert_{\rho}^2$
    \end{itemize}
\end{theorem}
\begin{proof}
    Define $\lambda=\cfrac{1}{n}\sum_{i=1}^n\phi_i, \tilde\phi_j=\phi_j-\lambda$.
    \begin{itemize}
        \item[(i)] $\tilde\phi_j-\tilde\phi_i=\phi_j-\phi_i$, so $\nabla_G\tilde\phi=\nabla_G\phi$.
        \item[(ii)] $\Vert\tilde\phi\Vert_{l_2}^2=\sum_{j=1}^n(\phi_j-\lambda)^2=\Vert\phi\Vert_{l_2}^2-n\lambda^2\leq\Vert\phi\Vert_{l_2}^2$, equality holds iff $\lambda=0$.
        \item[(iii)] $\sum_{j=1}^n\tilde\phi_j=\sum_{j=1}^n(\phi_j-\lambda)=\lambda n-n\lambda=0$.
        \item[(iv)] If $\phi_i=\lambda\quad\forall i$, then $\tilde\phi=0$. Suppose there exists $i_0$ such that $\phi_{i_0}\neq\lambda$. \\Let $\beta_j=\cfrac{\phi_j-\lambda}{\sqrt{\sum_{i=1}^n(\phi_i-\lambda)^2}}=\cfrac{\tilde\phi_j}{\Vert\tilde\phi\Vert_{l_2}}$. Then $\sum_{j=1}^n\beta_j=0, \sum_{j=1}^n\beta_j^2=1$. 
        Thus \[\gamma_{\text{Poincar\'e}}(\rho)\leq\cfrac{1}{2}\sum_{(i,j)\in E}g(\rho_i,\rho_j)\omega_{ij}(\beta_j-\beta_i)^2=\cfrac{\Vert\nabla_G\tilde\phi\Vert_\rho^2}{\Vert\tilde\phi\Vert_{l_2}^2}\].
    \end{itemize}
    
\end{proof}
\begin{theorem}\label{thm:2.5}
    Let $\rho\in\P(G)$ be such that $\gamma_{\text{Poincar\'e}}(\rho)>0$. Given $f\in\R^n$ such that $\sum f_i=0$, there exists $\phi\in\R^d$ such that $f=-\div_\rho(\nabla_G\phi)$. Further we have $\
    \Vert\nabla_G\phi\Vert^2_\rho\leq\Vert v\Vert_\rho^2$ when $f=-\div_\rho(v)$.
\end{theorem}
\begin{proof}
    Define $F(\phi)=\cfrac{1}{2}\Vert\nabla_G\phi\Vert^2_\rho-(f,\phi)_{l_2}$. Let 
    \[\inf\{F(\phi):\phi\in\R^n\}=\lim_{k\to\infty}F(\phi_k)\]. From Theorem \ref{thm:Poincare} let $(\tilde\phi_k)_k$ be such that $\Vert\nabla_G\tilde\phi_k\Vert_\rho^2=\Vert\nabla_G\phi_k\Vert_\rho^2$ \\Then \[(f, \phi_k-\tilde\phi_k)_{l_2}=\sum_{i=1}^nf_i(\phi_k-\tilde\phi_k)_i=\sum_{i=1}^n\lambda f_i=0\] and so $\lim_{k\to\infty}F(\phi_k)=\lim_{k\to\infty}F(\tilde\phi_k)$.
    
    From Theorem \ref{thm:Poincare} \[\Vert\tilde\phi\Vert_{l_2}\leq\frac{\Vert\nabla_G\phi\Vert_\rho}{\sqrt{\gamma_{\text{Poincar\'e}}(\rho)}}\]

    Thus \begin{align*}
        1=1+F(0)\geq F(\tilde\phi_k)&=\frac{1}{2}\Vert\nabla_G\tilde\phi_k\Vert^2_\rho-(f,\tilde\phi_k)_{l_2} \\
        &\geq \frac{1}{2}\Vert\nabla_G\tilde\phi_k\Vert_\rho^2-\Vert f\Vert_{l_2}\Vert\tilde\phi_k\Vert_{l_2} \\
        &\geq \frac{1}{2}\Vert\nabla_G\tilde\phi_k\Vert_\rho^2-\Vert f\Vert_{l_2}\frac{\Vert\nabla_G\phi_k\Vert_\rho}{\sqrt{\gamma_{\text{Poincar\'e}}}(\rho)}
    \end{align*}

    Hence \[\Vert\nabla_G\tilde\phi_k\Vert_\rho^2-2\Vert f\Vert_{l_2}\frac{\Vert\nabla_G\phi_k\Vert_\rho}{\sqrt{\gamma_{\text{Poincar\'e}}(\rho)}}-2\leq 0\]

    Thus we have 
    \[\frac{\Vert f\Vert_{l_2}}{\sqrt{\gamma_{\text{Poincar\'e}(\rho)}}}-\sqrt{\frac{\Vert f\Vert_{l_2}^2}{\gamma_{\text{Poincar\'e}}(\rho)}+8}\leq\Vert\nabla_G\tilde\phi_k\Vert_\rho^2\leq\frac{\Vert f\Vert_{l_2}}{\sqrt{\gamma_{\text{Poincar\'e}(\rho)}}}+\sqrt{\frac{\Vert f\Vert_{l_2}^2}{\gamma_{\text{Poincar\'e}}(\rho)}+8}\] and 

    \[\sup_{k}\Vert\nabla_G\tilde\phi_k\Vert_\rho\leq \frac{\Vert f\Vert_{l_2}}{\sqrt{\gamma_{\text{Poincar\'e}(\rho)}}}+\sqrt{\frac{\Vert f\Vert_{l_2}^2}{\gamma_{\text{Poincar\'e}}(\rho)}+8}\]

    Thus \[\Vert\tilde\phi_k\Vert_{l_2}\leq \frac{1}{\sqrt{\gamma_{\text{Poincar\'e}(\rho)}}}\bigg(\frac{\Vert f\Vert_{l_2}}{\sqrt{\gamma_{\text{Poincar\'e}(\rho)}}}+\sqrt{\frac{\Vert f\Vert_{l_2}^2}{\gamma_{\text{Poincar\'e}}(\rho)}+8}\bigg)\]

    Since $\tilde\phi_k$ is bounded, $\tilde\phi_k\to \phi$ up to some subsequence $(k_l)_{l=1}^{\infty}$.

    Thus \[F(\phi)=\lim_{l\to\infty}F(\tilde\phi_{k_l})=\lim_{l\to\infty}\frac{1}{2}\sum_{(i,j)\in E} \omega_{ij}g(\rho_i,\rho_j)((\phi_{k_l})_j-(\phi_{k_l})_i)^2-\sum_{i=1}^{n}f_{i}(\phi_{k_l})_i\]

    Therefore $F(\phi)=\inf\{F(\phi):\phi\in\R^n\}$.

    Let $\psi\in\R^n$. Then $h(\epsilon):=F(\phi+\epsilon\psi)\geq F(\phi)=h(0) \quad \forall \epsilon\in\R$.

    \begin{align*}
        h(\epsilon)&=\frac{1}{2}\Vert\nabla_G\phi+\epsilon\nabla_G\psi\Vert_\rho^2-(\phi+\epsilon\psi,f) \\ 
        &=\frac{1}{2}\Vert\nabla_G\phi\Vert_\rho^2+\epsilon(\nabla_G\phi, \nabla_G\psi)_\rho+\frac{\epsilon^2}{2}\Vert\nabla_G\psi\Vert^2_\rho-(\phi,f)-\epsilon(\psi, f) \\
        h'(0)&=(\nabla_G\phi,\nabla_G\psi)_\rho-(\psi,f)=-(\div_\rho(\nabla_G\phi)+f,\psi)=0
    \end{align*}

    Thus we get $f=-\div_\rho(\nabla_G\phi)$.

    Assume now that $v\in\mathbb{S}^{n\times n}$ also satisfies $f=-\div_\rho(v)$. Then 
    \begin{align*}
        \Vert v\Vert_\rho^2&=\Vert v-\nabla_G\phi+\nabla_G\phi\Vert_\rho^2 \\
        &=\Vert v-\nabla_G\phi\Vert^2_\rho+2(\div_\rho(v)-\div_\rho\nabla_G\phi,\phi)+\Vert\nabla_G\phi\Vert^2_\rho \\
        &\geq \Vert \nabla_G\phi\Vert_\rho^2
    \end{align*}
    where equality holds iff $\Vert v-\nabla_G\phi\Vert_\rho^2=0$.
\end{proof}
\begin{corollary}
    If $\sigma\in C^1([0,T],\P(G))$ then there exists $\phi\in C([0,T],\R^n)$ such that \\$\dot{\sigma}+\div_{\sigma}(\nabla_G\phi)=0$. 
    
    If $v\in C([0,T],\mathbb{S}^{n\times n})$ is another velocity for $\sigma$, then $\Vert\nabla_G\phi(t)\Vert^2_{\sigma(t)}\leq \Vert v(t)\Vert_{\sigma(t)}^2$.
\end{corollary}
\begin{proof}
    Let $f=\dot\sigma(t)$ in \ref{thm:2.5}
.\end{proof}

\noindent
\begin{definition}\textbf{(Continuity Equation)}
Let $\sigma\in C^1([0,T],\P(G))$ and $m\in C([0,T],
\mathbb{S}^{n\times n})
$. Assume

\begin{itemize}
    \item[(1)] \label{cont:1}
    $\dot\sigma(t)+\nabla_G\cdot m(t)=0\quad \forall t\in [0, T)$ 

    Then for every $\phi\in C^1([0,T]\to\R^n)$, we have 

    \begin{align*}
        0&=\int_0^T(\phi(t),\dot\sigma(t)+\nabla_G\cdot m(t))dt\\
        &=-\int_0^T(\dot\phi(t),\sigma(t))dt+(\phi(T),\sigma(T))-(\phi(0),\sigma(0))-\int_0^T(\nabla_G\phi(t),m(t))_{\sigma(t)} dt
    \end{align*}
    \item[(2)] \label{cont:2} 
    \begin{align*}
        0&=(\phi(T),\sigma(T))-(\phi(0),\sigma(0))-\int_0^T((\dot\phi(t),\sigma(t))+(\nabla_G\phi(t),m(t)))_{\sigma(t)}dt\quad \forall \phi\in (C^1[0,T],\R^n)
    \end{align*}
    For (2) to make sense, we need
    \item[(3)]\label{cont:3}
    \[\int_0^T|\sigma(t)|dt<+\infty\] and 
    \[\int_0^T|m(t)|dt<+\infty\] 
    where $\sigma\in L(0,T;\P(G))$ and $m\in L(0,T;\mathbb{S}^{n\times n})$
\end{itemize}
When (2) and (3) hold, we say that (1) is satisfied in the sense of distribution.
\end{definition}
Assume, we can find $v:[0,T]\to\mathbb{S}^{n\times n}$ such that $g(\sigma_i, \sigma_j)v_{ij}=m_{ij}$.

Note that (2) means $\dot\sigma+\div_\sigma(v)=0$ in the sense of distribution.
\[
(\nabla_G\phi,m)_{\sigma}=\frac{1}{2}\sum_{(i,j)\in E}(\nabla_G\phi)_{ij}g(\sigma_i,\sigma_j)v_{ij}=(\nabla_G\phi,v)_\sigma
\]

The kinetic energy at time $t$ is 
\begin{align*}
    &\frac{1}{2}\Vert v(t)\Vert_{\sigma(t)}^2\\
    &=\frac{1}{4}\sum_{(i,j)\in E}g(\sigma_i,\sigma_j)v_{ij}^2\\
    &=\frac{1}{4}\sum_{\substack{(i,j)\in E\\ g(\sigma_i,\sigma_j)\neq 0}}g(\sigma_i,\sigma_j)\frac{m_{ij}^2}{g(\sigma_i,\sigma_j)^2} \\
    &= \frac{1}{4}\sum_{\substack{(i,j)\in E\\g(\sigma_i,\sigma_j)\neq 0}}F(g(\sigma_i,\sigma_j),m_{ij})
\end{align*}
if we set \[F(a,b)=\begin{cases}
    \frac{|b|^2}{a} & a>0 \\
    0 & a=b=0 \\
    +\infty & a=0, b\neq 0
\end{cases}\]

\noindent
\begin{definition}\textbf{(Wasserstein metric on $\P(G)$)}

If $\rho, \rho^*\in \P(G)$, set 
\begin{align*}
     \W^2(\rho,\rho^*) 
    &= \inf_{(\sigma, m)}\bigg\{\int_0^1\frac{1}{2}\sum_{(i,j)\in E}F(g(\sigma_i,\sigma_j),m_{ij})dt:\sigma\in C([0,T],\P(G)), m\in C([0,T],
\mathbb{S}^{n\times n}, \\
&\dot\sigma+\nabla_G\cdot m=0, \sigma(0)=\rho^0, \sigma(1)=\rho^*\bigg\}\\
    &=
    \inf_{(\sigma,v)}\bigg\{\int_{0}^1 \Vert v(t)\Vert_{\sigma(t)}^2dt: \sigma\in C([0,T],\P(G)), v:[0,T]\to\mathbb{S}^{n\times n} \text{ is Borel}, \\
    &\dot{\sigma}+\div_{\sigma}(v)=0, \sigma(0)=\rho, \sigma(1)=\rho^*\bigg\}
\end{align*}
\end{definition}
\begin{remark}
    For $\W^2(\rho, \rho^{*})$ to be a metric, we need to assume that 
    \begin{equation}
        \int_0^1\frac{dr}{\sqrt{g(r,1-r)}}<+\infty
    \end{equation}
\end{remark}
\begin{lemma}\label{lemma: concat}
    Assume \begin{align*}
        &\sigma:[0,1]\to\P(G), \sigma(0)=\rho^0, \sigma(1)=a \\
    &\tilde\sigma:[0,1]\to\P(G), \tilde\sigma(0)=a, \tilde\sigma(1)=\rho^1
    \end{align*}
    and \[\int_0^1\Vert v\Vert_\sigma^2 dt<+\infty, \int_0^1\Vert\tilde v\Vert^2_{\tilde\sigma} dt<+\infty\]
    Define \[r(s)=\begin{cases}
        \sigma(2s) & 0\leq s\leq \frac{1}{2} \\
        \tilde\sigma(2s-1) & \frac{1}{2}\leq s \leq 1
    \end{cases}\]
    \[w(s)=\begin{cases}
        2v(2s) & 0\leq s<\frac{1}{2} \\
        2\tilde{v}(2s-1) & \frac{1}{2}\leq s\leq1
    \end{cases}\]
    If $(\sigma, v)$ and $(\tilde\sigma, \tilde{v})$ satisfies the continuity equation in the sense of distributions, so does $(r, w)$. Furthermore \[\int_0^1\Vert w\Vert^2_r dt=2\bigg(\int_0^1\Vert v\Vert_{\sigma}^2dt+\int_0^1\Vert\tilde{v}\Vert_{\tilde\sigma}^2\bigg)dt\].
\end{lemma}
\begin{proof}
    Let $\psi\in C^1([0,1], \R^n)$ where \[
    \psi(s)=\begin{cases}
        \phi(2s) & 0\leq s<\frac{1}{2} \\
        \tilde\phi(2s-1) & \frac{1}{2}\leq s\leq 1
    \end{cases}
    \] for some $\phi,\tilde\phi\in C^1([0,1],\R^n)$ such that $\phi(1)=\tilde\phi(0)$.

    Since $(\sigma, v), (\tilde\sigma, \tilde{v})$ satisfies the continuity equation,
    \begin{align*}
        (\phi(1), a)-(\phi(0), \rho^0)-\int_0^1((\dot\phi(t),\sigma(t))+(\nabla_G\phi(t),v(t))_{\sigma}dt= 0 \\
        (\tilde\phi(1), \rho^1)-(\tilde\phi(0), a)-\int_0^1((\dot{\tilde\phi}(t),\tilde\sigma(t))+(\nabla_G\tilde\phi(t),\tilde{v}(t))_{\tilde\sigma}dt= 0
    \end{align*}
    Adding the two equations we get
    \begin{equation*}\begin{split}
        (\psi(1), r(1))-(\psi(0), r(0))-\int_0^1((\dot\phi(t),\sigma(t))+(\nabla_G\phi(t),v(t))_{\sigma}dt\\-\int_0^1((\dot{\tilde\phi}(t),\tilde\sigma(t))+(\nabla_G\tilde\phi(t),\tilde{v}(t))_{\tilde\sigma}dt= 0\end{split}
    \end{equation*}
    Now 
    \begin{align*}
        &\int_0^1((\dot\phi(t),\sigma(t))+(\nabla_G\phi(t),v(t))_{\sigma(t)}dt \\
        &=\int_0^{1/2}((\dot\phi(2s),\sigma(2s))+(\nabla_G\phi(2s),v(2s))_{\sigma(2s)}\cdot2ds \\
        &=\int_0^{1/2}\bigg(\bigg(\frac{1}{2}\dot\psi(s),r(s)\bigg)+\bigg(\nabla_G\psi(s),\frac{1}{2}w(s)\bigg)_{r(s)}\cdot2ds \\
        &=\int_0^{1/2}((\dot\psi(s),r(s))+(\nabla_G\psi(s),w(s))_{r(s)}ds
    \end{align*}
    Similarly \[\int_0^1((\dot{\tilde\phi}(t),\tilde\sigma(t))+(\nabla_G\tilde\phi(t),\tilde{v}(t))_{\tilde\sigma}dt=\int_{1/2}^1((\dot\psi(s),r(s))+(\nabla_G\psi(s),w(s))_{r(s)}ds\]
    Therefore \[\begin{split}(\psi(1), r(1))-(\psi(0), r(0))-\int_0^1((\dot\phi(t),\sigma(t))+(\nabla_G\phi(t),v(t))_{\sigma(t)}dt\\-\int_0^{1}((\dot\psi(s),r(s))+(\nabla_G\psi(s),w(s))_{r(s)}ds=0\end{split}\]
    Further \[\int_0^1|r(t)|dt=\int_0^{1/2}|\sigma(t)|dt+\int_{1/2}^1|\tilde\sigma(t)|dt\leq \int_0^1|\sigma(t)|dt+\int_0^1|\tilde\sigma(t)|dt<+\infty\]

    \begin{align*}
        \int_0^1\Vert w\Vert_r^2dr&
        =\int_0^1\frac{1}{2}\sum_{(i,j)\in E}g_{ij}(r)w_{ij}^2dt \\
        &=\int_0^{1/2}\Vert w\Vert_r^2dr=\int_0^1\frac{1}{2}\sum_{(i,j)\in E}g_{ij}(r)w_{ij}^2dt+\int_{1/2}^1\Vert w\Vert_r^2dr=\int_0^1\frac{1}{2}\sum_{(i,j)\in E}g_{ij}(r)w_{ij}^2dt \\
        &=\int_0^{1/2}\sum_{(i,j)\in E}g_{ij}(\sigma(2t))\cdot 2v_{ij}(2t)^2dt+\int_{1/2}^1\sum_{(i,j)\in E}g_{ij}(\tilde\sigma(2t-1))\cdot 2\tilde{v}_{ij}(2t-1)^2dt \\
        &=\int_0^{1}\sum_{(i,j)\in E}g_{ij}(\sigma(s))v_{ij}(s)^2ds+\int_0^{1}\sum_{(i,j)\in E}g_{ij}(\tilde\sigma(s))\tilde{v}_{ij}(s)^2ds \\
        &=2\bigg(\int_0^1\Vert v\Vert_\sigma^tdt+\int_0^1\Vert\tilde{v}\Vert_{\tilde\sigma}^2dt\bigg)
    \end{align*}
\end{proof}
\begin{theorem}
    For any $\rho^0,\rho^1\in\P(G)$ there exists $(\sigma, m)$ which is a solution to the continuity equation such that $\sigma(0)=\rho^0, \sigma(1)=\rho^1, \int_0^1\sum_{(i,j)\in E}F(g(\sigma_i,\sigma_j),m_{ij})dt<+\infty$. 
\end{theorem}
\begin{proof}
    From \ref{lemma: concat} it is enough to show the case when \[\rho^1=\begin{pmatrix}
        0 \\
        \vdots \\
        0 \\
        1
    \end{pmatrix}\].
    First we show that it is enought to show the case when $n=2$. If $\rho\in \P(G)$, define
    
    \[V[\rho]=\{i\in\{1, ..., n-1\}:\rho_i>0\}\].
    Suppose there is a path from $\rho^0$ to $\mu^1$ such that $\#V[\rho^0]>\#V[\mu^1]$. Then $\#V[\mu^1]< n-1$. Similarly, suppose there is a path from $\mu^i$ to $\mu^{i+1}$ such that $\#V[\mu^i]>\#V[\mu^{i+1}]$. Then $\#V[\mu^i]< n-i$.
    If $\#V[\mu^l]=0$, then $\mu_l=\rho^1$. Otherwise continue with $\mu^{l+1}$.

    Next show the case for $n=2$. Let $\rho^0=\begin{pmatrix}
        \rho_1^0 \\
        \rho_2^0
    \end{pmatrix}, \rho^1=\begin{pmatrix}
        \rho_1^1 \\
        \rho_2^1
    \end{pmatrix}$. Also $\omega_{12}>0$.

    If $\rho_1^0=\rho_1^1$, then $\rho^0=\rho^1$. Thus we can choose $\sigma(t)=\rho^0, m(t)=0\quad \forall t$.

    Else, WLOG assume $\rho_1^0<\rho_1^1$.
    Define \[G(t)=\int_0^t\frac{dr}{\sqrt{g(r,1-r)}}\]. Then $G'(r)=\frac{1}{\sqrt{r, 1-r}}>0$. Thus $G$ is monotonically increasing, and $G:[0,1]\to [0,G(1)]$ is a bijection.

    Let \begin{align}\label{eq:sigma}
        &\sigma_1(t)=G^{-1}(at+G(\rho_1^0)) \quad \text{where } a=G(\rho_1^1)-G(\rho_1^0) \\
        &\sigma_2(t)=1-\sigma_1(t)
    \end{align} 
    Let \begin{equation}\label{eq:m}
        m_{21}=-\frac{\dot\sigma_1}{\sqrt{\omega_{12}}}=-m_{12}
    \end{equation} so $\dot\sigma+\nabla_G\cdot m=0$.

    Then \begin{equation}
        \sigma(0)=\begin{pmatrix}
            \sigma_1(0) \\
            \sigma_2(0)
        \end{pmatrix} = \begin{pmatrix}
            G^{-1}(G(\rho_1^0)) \\
            1-G^{-1}(G(\rho_1^0))
        \end{pmatrix}=\begin{pmatrix}
            \rho_1^0 \\
            1-\rho_1^0
        \end{pmatrix}=\rho^0
    \end{equation} and \begin{equation}
        \sigma(1)=\rho^1
    \end{equation}
    From \eqref{eq:sigma} $G(\sigma_1(t))=at+G(\rho_1^0)$ and so $\dot\sigma_1G'(\sigma_1)=\cfrac{\dot\sigma_1}{\sqrt{g(\sigma_1,\sigma_2)}}a$.

    From \eqref{eq:m}, $\cfrac{m_{12}^2}{g(\sigma_1,\sigma_2)}=\cfrac{\dot\sigma_1^2}{\omega_{12}}\cfrac{1}{g(\sigma_1,\sigma_2)}=\cfrac{\dot\sigma_1a}{\omega_{12}\sqrt{g(\sigma_1,\sigma_2)}}=\dot\sigma_1\cfrac{a}{\omega_{12}\sqrt{g(\sigma_1,1-\sigma_1)}}$. 

    Thus \begin{align*}
        \int_0^1\frac{m_{12}^2}{g(\sigma_1,\sigma_2)}dt&=\int_0^1\frac{a\dot\sigma_1}{\omega_{12}\sqrt{g(\sigma_1, 1-\sigma_2)}}dt\\
        &=\int_{\sigma_1(0)}^{\sigma_1(1)}\frac{ads}{\omega_{12}\sqrt{g(s, 1-s)}}\leq \frac{a}{\omega_{12}}\int_0^1\frac{ds}{\sqrt{g(s, 1-s)}}<+\infty
    \end{align*}
 \end{proof}
\begin{corollary}
    $\W(\rho^0,\rho^1)<+\infty$ for all $\rho^0, \rho^1\in\P(G)$. 
\end{corollary}

\begin{lemma}\label{lem:2.7}
    Define $C_\omega=\sup_{(i,j)\in E}\sqrt{\omega_{ij}}$. Then $\Vert \dot\sigma(t)\Vert_{l\infty}\leq \sqrt{2}nC_{\omega}\W(\rho^0, \rho^1)$
\end{lemma}
\begin{proof}
Note that $g_{ij}(\rho)\leq \rho_i+\rho_j$.
    \begin{align*}
        &\Vert \div_\rho(v)\Vert_{l_{1}}=\Vert\sum_{(i,j)\in E}\sqrt{\omega_{ij}}g_{ij}(\rho)v_{ji}\Vert_{l_1} \leq C_\omega\Vert\sum_{(i,j)\in E} g_{ij}(\rho)v_{ji}\Vert_{l_1} \leq \sqrt{2}C_\omega\Vert v\Vert_\rho
    \end{align*}
    Thus 
    \begin{align*}
        \Vert\dot\sigma(t)\Vert_{l_{\infty}}\leq \Vert \div_\sigma(v)\Vert_{l_1}\leq \sqrt{2}nC_\omega\Vert v\Vert_{\sigma}\leq \sqrt{2}nC_\omega\W(\rho^0,\rho^1)
    \end{align*}
\end{proof}

\begin{theorem}
    $\W$ is a metric on $\P(G)$ provided that $\W(\rho,\rho^{*})<+\infty$ for all $\rho,\rho^{*}\in\P(G)$.
\end{theorem} 

\begin{proof}
    Symmetry is clear from the definition.

    Suppose $\rho=\rho^*$. Then $\W^2(\rho, \rho^*)=0$ is clear. 

    Suppose $\W^2(\rho, \rho^*)=0$. From Lemma \ref{lem:2.7}, $\Vert\dot\sigma(t)\Vert_{l_\infty}=0$, so $\rho=\rho^*$..

    Let $\bar{\rho}\in\P(G)$. Suppose 
    \begin{align*}
        &\W(\rho,\bar{\rho})+\W(\bar{\rho},\rho^*)\\
        &=\int_0^1\Vert v(t)\Vert_{\sigma(t)}^2dt+\int_0^1\Vert w(t)\Vert_{\phi(t)}^2dt \\
        &=\int_0^{\frac{1}{2}}\Vert v(2t)\Vert_{\sigma(2t)}^2 2dt+\int_{\frac{1}{2}}^1\Vert w(2t-1)\Vert_{\phi(2t-1)}^2 2dt
    \end{align*}

    Define \[\psi(t)=\begin{cases}
        \sigma(2t) & 0\leq t\leq \frac{1}{2} \\
        \phi(2t-1) & \frac{1}{2}\leq t\leq 1
    \end{cases}\] 
    \[u(t)=\begin{cases}
        v(2t) & 0\leq t\leq \frac{1}{2} \\
        w(2t-1) & \frac{1}{2}\leq t \leq 1
    \end{cases}\]
    Then 
    \begin{align*}
        &\W(\rho,\bar{\rho})+\W(\bar{\rho},\rho^*)\\
        &=2\int_0^1\Vert u(t)\Vert_{\psi(t)}^2dt\geq\int_0^1\Vert u(t)\Vert_{\psi(t)}^2dt\geq\W(\rho,\rho^*)
    \end{align*}
\end{proof}

\noindent
Let $v\in\mathbb{S}^{n\times n}$. We say that $v\in T^1_\rho\P(G)$ if there exists $(\phi_l)_{l}\subset\R^n$ such that \\$\lim_{l\to\infty}\Vert v-\nabla_G\phi_l\Vert_\rho=0$.

\begin{lemma}\label{lem:T2}
    Let $\rho\in\P(G)$ and $v\in\mathbb{S}^{n\times n}$. There exists a unique $v^*\in T^1_\rho\P(G)$ such that $\Vert v-v^{*}\Vert_\rho\leq \Vert v-w\Vert_\rho\quad\forall w\in T^1_\rho\P(G)$. Further, $(v-v^{*}, w)_\rho=0\quad \forall w\in T^1_\rho\P(G)$. Then we say $v^{*}=\proj_{T^1_\rho\P(G)}v$.
\end{lemma}
\begin{proof}
    Set \[\inf_{w\in T_\rho^1\P(G)}\Vert v-w\Vert_\rho^2=\lim_{l\to\infty}\Vert v-w^l\Vert^2_\rho\].

    \[1+\Vert v-0\Vert^2_\rho\geq\Vert v-w^l\Vert^2_\rho=\frac{1}{2}\sum_{(i,j)\in E}g_{ij}(\rho)(v_{ij}-w_{ij}^l)^2\]

    For each $(i,j)\in E$ such that $g_{ij}(\rho)>0$, we have 
    \[\frac{2}{g_{ij}(\rho)}(1+\Vert v\Vert_\rho^2)\geq (v_{ij}-w_{ij}^l)^2\].

    Hence $(w_{ij}^l)_l$ is bounded in $\R$, so it has a convergence subsequence. Since there are only finitely many $(i,j)\in E$, we can find a common subsequence $(l_k)_{k=1}^\infty$ such that $(w_{ij}^{l_k})_{k}$ converges to some $w_{ij}$ as $k\to\infty$. For $(i,j)\in E$, set $\phi^l$ such that $\phi_i^l-\phi_j^l=\cfrac{w^l_{ij}}{\sqrt{\omega_{ij}}}$. Else let $\phi^l_i=\phi^l_j=0$.

    Then $(w_{ij}^{l_k})_k\to w_{ij}$, so $(w^{l_k})_k=(\nabla_G\phi^{l_k})_k\to w$. Thus when $w^l\in T_\rho^1\P(G)$, we can assume that there exists $\phi^l\in\R^n$ such that $\Vert w^l-\nabla_G\phi^l\Vert_\rho<\cfrac{1}{l}$.

    Now set $v^*=\begin{cases}
        w_{ij} & (i,j)\in E, g_{ij}(\rho)>0 \\
        0
    \end{cases}$
    Then \[\Vert v^{*}-\nabla_G\phi^l\Vert_\rho^2=\frac{1}{2}\sum_{\substack{(i,j)\in E\\ g_{ij}(\rho)>0}} g_{ij}(\rho)(v^{*}_{ij}-(\nabla_G\phi^l)_{ij}^2)=\frac{1}{2}\sum_{\substack{(i,j)\in E\\ g_{ij}(\rho)>0}}g_{ij}(\rho)(w_{ij}-(\nabla_G\phi^{l})_{ij})^2\]

    We also have that \[\lim_{k\to\infty}\frac{1}{2}g_{ij}(\rho)(w_{ij}^{l_k}-(\nabla_G\phi^{l_k})_{ij})^2\leq\lim_{k\to\infty}\frac{1}{l_k^2}=0\]

    Therefore $\lim_{k\to\infty}\Vert v^{*}-\nabla_G\phi^{l_k}\Vert_\rho=0$, so $v^{*}\in T_\rho^1\P(G)$.

    Suppose there exists $w^{*}\in T_\rho^1\P(G)$ such that $\Vert v-w^{*}\Vert_\rho=\Vert v-v^{*}\Vert_\rho$. \\Then $v^{*}_{ij}=w^{*}_{ij}\quad \forall (i,j)\in E, g_{ij}(\rho)\neq 0$, so $v^{*}=w^{*}$.

    Define $f(t)=\Vert v^{*}+tw-v\Vert_\rho^2$.
    We have $f(0)\leq f(t)\quad \forall t\in \R$, since $v^{*}+tw\in T_\rho^1 \P(G)$.
    $f(t)=\Vert v^{*}-v\Vert_{\rho}^2+2(v^{*}-v, w)_\rho t+t^2\Vert w\Vert_\rho^2$, so
    $f'(0)=2(v^{*}-v, w)_\rho=0$, thus \\$(v^*-v, w)_\rho=0$.
\end{proof}
\begin{definition}\textbf{(Tangent space of $\P(G)$)}
    Let $\rho\in\P(G)$. If $v, \bar{v}\in\mathbb{S}^{n\times n}$ are such that $g_{ij}(\rho)v_{ij}=g_{ij}(\rho)(\bar{v}_{ij})$, we say that $v=\bar{v}$ a.e. \\In fact, \[\Vert v-\bar{v}\Vert^2_\rho=\frac{1}{2}\sum_{(i,j)\in E}(v_{ij}-\bar{v}_{ij})^2g_{ij}(\rho)=0\]

Thus we define $[v]=\{\bar{v}\in\mathbb{S}^{n\times n}:v=\bar{v} \text{ a.e.}\}$ and $\mathbb{H}_p=\{[v]:v\in\mathbb{S}^{n\times n}\}$.
Define $\Pi_\rho:\mathbb{S}^{n\times n}\to \mathbb{S}^{n\times n}$ where $\Pi_\rho(w)=\argmin_{v} \{\Vert w-v\Vert_\rho :\div_\rho(w-v)=0\}$. Let $T_\rho^2\P(G)=\Pi_\rho(\mathbb{S}^{n\times n})$. In Lemma \ref{lem:T2} we showed that $\Pi_\rho$ is well-defined and in Theorem \ref{thm:tangent} we show that $T_\rho^1\P(G)=T_\rho^2\P(G)$.

\end{definition}
\begin{theorem}{\label{thm:tangent}}
    $T_\rho^1\P(G)=T_\rho^2\P(G)$ 
\end{theorem}

\begin{proof}
    Let $w\in\mathbb{S}^{n\times n}$. Then 
    
    \begin{align*}
        &\Pi_\rho(w)=\proj_{T^1_\rho\P(G)}w\\
        &\Longleftrightarrow\Pi_\rho(w)\in T_1^\rho\P(G) \text{ and } (w-\Pi_\rho(w),v)_\rho=0\quad \forall v\in T_\rho^1\P(G)\\
        &\Longleftrightarrow \Pi_\rho(w)\in T_1^\rho\P(G) \text{ and } (w-\Pi_\rho(w), \nabla_G\phi)_\rho=0\quad\forall \phi\in\R^n
    \end{align*}

    By definition the image of $\Pi_\rho$ is contained in $T^1_\rho\P(G)$, so $T_\rho^2\P(G)\subset T_\rho^1\P(G)$.

    If $v\in T_\rho^1\P(G)$, then $v=\Pi_\rho(v)$ and so, $v\in T_\rho^2\P(G)$.
\end{proof}
From now on we will denote $T_\rho^1\P(G)=T_\rho^2\P(G)$ as $T_\rho\P(G)$.
\begin{remark}
    From Theorem \ref{thm:2.5}, if $f\in R^n$ and $\sum f_i=0$, then for all $\rho$ such that $\gamma_{\text{Poincar\'e}}(\rho)>0$, there exists $\phi\in R^n$ such that $f=-\div_\rho(\nabla_G\phi)$.

    If $\rho\in\P(G)$, then there exists $v\in T_\rho\P(G)$ such that $f=-\div_\rho(v)$. 
\end{remark}

\noindent
\begin{definition}\textbf{(Fr\'echet derivative)}

Let $\rho\in\P(G)$ and $\F:\P(G)\to\R$. We say that $\F$ has a Fr\'echet derivative at $\rho$ if there exists $f\in\R^n$ such that $\sum f_i=0$ and for every $\bar{\rho}\in\P(G)$,
\begin{equation}
    \label{frechet}\lim_{t\to 0}\frac{\F((1-t)\rho+t\bar{\rho})-\F(\rho)}{t}=(f,\bar{\rho}-\rho)
\end{equation}

Note that $(1-t)\rho+t\bar{\rho}=\rho+t(\bar{\rho}-\rho)$. We denote $f=\frac{\delta \F}{\delta\rho}(\rho)$ and call it the Fr\'echet derivative at $\rho$.
\end{definition}
\begin{lemma} (Lemma 3.13 in \cite{graph})
    When Fr\'echet derivative exists, it is unique.
\end{lemma}

\begin{remark}
    Let $v\in T_\rho\P(G)$ such that $f=-\div_\rho(v)$. If (\ref{frechet}) holds,
    \begin{equation}
        \frac{d}{dt}\F(\rho_t)|_{t=0}=(f,\bar{\rho}-\rho)-(-\div_\rho(v),\bar{\rho}-\rho)=(v,\nabla_G(\bar\rho-\rho))_\rho=(v-\bar{v})_\rho
    \end{equation}
    where $\bar{v}\in T_\rho\P(G)$ and $\rho_t=\rho+t(\bar{\rho}-\rho)$.
\end{remark}

\noindent
\begin{definition}\textbf{(Wasserstein gradient)}

We say that $\F$ s differentiable in the Wasserstein sense at $\rho$ if there exists $v\in T_\rho\P(G)$ and $c>0$ such that the following holds: for any $\epsilon>0$, there exists $\gamma>0$ such that if $\bar{\rho}\in\P(G), \bar{v}\in T_\rho\P(G)$,
\begin{equation}
    \Vert\bar\rho-\rho\Vert_{l_1}<\delta \Longrightarrow |\F(\bar\rho)-\F(\rho)-(v,\bar{v})_\rho|\leq \epsilon\W(\rho,\bar\rho)+c\Vert\bar\rho-\rho+\div_\rho(\bar{v})\Vert_{l_{1}}
\end{equation}
\end{definition}
\begin{theorem}
    Assume $\rho\in\P_0(G)$. There is at most one $v\in T_\rho\P(G)$ satisfying the property of Wasserstein differential. We set $v=\nabla_\W\F(\rho)$ and call $v$ the Wasserstein gradient of $\F$ at $\rho$.
\end{theorem}
\begin{proof}
    Fix $\epsilon>0$. Let $\delta>0$ be such that $v, \tilde{v}\in T_\rho\P(G)$ satisfies the condition for $\F$ to be Wasserstein differentiable. Also fix $(i,j)\in E$.

    Define \begin{equation*}\bar{v}_{ji}=\begin{cases}
        \cfrac{\sqrt{\omega_{ij}}\alpha}{g_{ij}(\rho)} & g_{ij}(\rho)\neq 0\\
        0 & g_{ij}(\rho)= 0
    \end{cases}\end{equation*} where $\alpha\in\R$.

    Define $\sigma(t)=\rho-t\div_\rho(\bar{v})$. Then $\sigma(t)=\rho-\div_\rho(\bar{v})$, and set this to $\bar\rho$. Then $\dot\sigma+\div_\rho(\bar{v})=0$.

    For $\alpha<<1$, $\Vert\bar\rho-\rho\Vert_{l_1}<\delta$.

    Thus \begin{align*}
        |\F(\bar\rho)-\F(\rho)-(v,\bar{v})_\rho|, |\F(\bar\rho)-\F(\rho)-(\tilde{v},\bar{v})_\rho|&\leq \epsilon\W(\rho,\bar\rho)+c\Vert\bar\rho-\rho+\div_\rho(\bar{v})\Vert_{l_1}\\
    &=\epsilon\W(\rho,\bar\rho)
    \end{align*}
    Hence \begin{equation}\label{eq:2.12}
        |(v-\tilde{v},\bar{v})_\rho|\leq 2\epsilon \W(\rho, \bar{\rho})
    \end{equation}
    
    Now define $v_{ji}^*=\cfrac{g_{ij}(\rho)}{g_{ij}(\sigma)}\bar{v}_{ji}$.
    Then
    \begin{align*}
        \langle\div_\sigma(v^*)\rangle_i=\bigg\langle\sum_{j\in N(i)}\sqrt{\omega_{ij}}g_{ij}(\sigma)v^*_{ji}\bigg\rangle_i=\bigg\langle\sum_{j\in N(i)}\sqrt{\omega_{ij}}g_{ij}(\rho)\bar{v}_{ji}\bigg\rangle_i=\langle\div_\rho(\bar{v})\rangle_i
    \end{align*}
    $\dot\sigma+\div_\sigma(v^*)=0$, so 
    \begin{align*}
        \W^2(\rho, \bar\rho)\leq \int_0^1\Vert v^{*}(t)\Vert_{\sigma(t)}^2dt = \int_0^1\sum_{(i,j)\in E}\frac{1}{2}\frac{\omega_{ij}\alpha^2}{g_{ij}(\sigma)}dt
    \end{align*}
    Further \begin{align*}
        |(v-\tilde v, \bar{v})_\rho|=\bigg|\frac{1}{2}\sum_{(i,j)\in E}\sqrt{\omega_{ij}}\alpha(v-\tilde{v})_{ij}\bigg|
    \end{align*}
    From \eqref{eq:2.12} \[\sum\sqrt{\omega_{ij}}\alpha|v-\tilde{v}|_{ij}\leq 2\sqrt{\omega_{ij}}\alpha\epsilon\sqrt{\int_0^1\cfrac{1}{g_{ij}(\sigma)}dt}\]. Since $\epsilon>0$ is arbitrary, $(v-\tilde{v})_{ij}=0$, so $v=\tilde{v}$.
\end{proof}
\begin{lemma}(Lemma 3.5 in \cite{graph})\label{lem:c}
    For every $\epsilon_0>0$, if $\rho,\bar\rho\in\P_{\epsilon_0}(G)$, there exists $c>0$ such that $\sqrt{\epsilon_0}\W(\rho,\bar\rho)\leq c\Vert\rho-\bar{\rho}\Vert_{l_1}$
\end{lemma}
\begin{theorem}
    Suppose $\F$ has a Wasserstein differential at $\rho$. If $\rho\in \P_0(G)$ then $\F$ has a Fr\'echet differential at $\rho$.
\end{theorem}
\begin{proof}
    Since $\rho\in\P_0(G)$ and $v=\nabla_\W\F(\rho)\in T_\rho\P(G)$, there exists $\phi\in\R^n$ such that 
    \begin{equation}
        v=\nabla_G\phi
    \end{equation}
    Let $f_i=\phi_i-\cfrac{1}{n}\sum_{j=1}^n\phi_j$. Then $\sum f_i=0$ and $v=\nabla_G f$, Since $\F$ has a Wasserstein differential at $\rho$, there exists $c>0$ such that for every $\epsilon>0$, there exists $\delta>0$ such that if $\rho^{*}\in\P(G)$ and $v^{*}\in T_\rho\P(G)$ then
    \begin{equation}\label{eq:wass}
        \Vert \rho^*-\rho\Vert_{l_1}<\delta\Longrightarrow |\F(\rho^*)-\F(\rho)-(v,v^*)_\rho|\leq\epsilon\W(\rho^*, \rho)+c\Vert\rho^*-\rho+\div_\rho(v^*)\Vert_{l_1}
    \end{equation}
    Let $\bar\rho\in\P(G)$. Since $\rho\in\P_0(G)$, there exists $\epsilon_0>0$ such that 
    \begin{equation}\label{eq:epsilon0}
        \rho\in\P_{2\epsilon_0}(G)
    \end{equation}
    Let $\rho_t=(1-t)\rho+t\bar\rho$.

    Then $\Vert\rho_t-\rho\Vert_{l_1}+\Vert t(\bar\rho-\rho)\Vert_{l_1}=2t\Vert\bar\rho-\rho\Vert_{l_1}$ so for $0<t<<1$, \eqref{eq:epsilon0} implies that $\rho_t\in\P_{\epsilon_0}(G)$. Note that $\sum(\bar\rho_i-\rho_i)=0$, and so there exists $\bar{v}\in T_\rho\P(G)$ such that 
    \begin{equation}\label{eq:rho}
        \bar\rho-\rho=-\div_\rho(\bar{v})
    \end{equation}
    By \eqref{eq:wass} if $0<t<<1$, $\Vert\rho_t-\rho\Vert_{l_1}<\delta$.

    Using Lemma \ref{lem:c} and \eqref{eq:rho},
    \begin{align*}
        |\F(\rho_t)-\F(\rho)-(v,t\bar{v})_\rho|\leq \epsilon\W(\rho_t, \rho)+c\Vert\rho_t-\rho+\div_\rho(t\bar{v})\Vert_{l_1} \\
        \leq \epsilon\frac{1}{\sqrt{\epsilon_0}}\Vert\rho_t-\rho\Vert_{l_1}+tc\Vert(\bar\rho-\rho)+\div_\rho(\bar{v})\Vert_{l_1}=\epsilon\frac{1}{\sqrt{\epsilon_0}}\Vert\rho_t-\rho\Vert_{l_1}
    \end{align*}
    Therefore 
    \[\bigg|\frac{\F(\rho_t)-\F(\rho)}{t}-(v,\bar{v})_\rho\bigg|\leq\frac{\epsilon}{\sqrt{\epsilon_0}}\Vert\bar\rho-\rho\Vert_{l_1}\]

    Hence 
    \[\limsup_{t\to 0^+}\bigg|\frac{\F(\rho_t)-\F(\rho)}{t}-(v,\bar{v})_\rho\bigg|\leq\frac{\epsilon}{\sqrt{\epsilon_0}}\Vert\bar\rho-\rho\Vert_{l_1}\]

    Using \eqref{eq:rho},
    \[
    \limsup_{t\to 0^+}\bigg|\frac{\F(\rho_t)-\F(\rho)}{t}-(f,\bar\rho-\rho)_\rho\bigg|\leq\frac{\epsilon}{\sqrt{\epsilon_0}}\Vert\bar\rho-\rho\Vert_{l_1}
    \]
    In conclusion,
    \[
    \lim_{t\to 0^+}\frac{\F(\rho_t)-\F(\rho)}{t}-(v,\bar{v})_\rho=0
    \]
\end{proof}

\begin{theorem}\label{thm:2.14}(Lemma 3.14 in \cite{graph})
    Suppose $\F$ is Fr\'echet continuously differentiable at $\rho$. Then $\F$ has a Wasserstein differential and $\nabla_\W\F(\rho)=\nabla_G(\frac{\delta\F}{\delta\rho})(\rho)$.
\end{theorem}

\section{Solutions to HJE for a particular $g$}

In this section, we assume that 
\begin{equation}
    g(s,t)=\begin{cases}
        \frac{s-t}{\log s-\log t} & s\neq t, st\neq 0 \\
        0 & st=0 \\
        t & s=t, st\neq 0
    \end{cases}
\end{equation}

Further, define 
    \begin{equation}
        A_{ij}=\begin{cases}
    \omega_{ij} & j\in N(i) \\
    0 & j\neq N(i), j\neq i \\
    -\sum_{k\in N(i)}\omega_{ik} & j=i
    \end{cases}
    \end{equation}

Note that $A$ is symmetric.

\begin{lemma}\label{lem:3.1}
    For any $t\geq 0$, $e^{tA}$ is a transition probability matrix.
\end{lemma}
\begin{proof}
    \begin{equation*}
        \sum_{j=1}^n\bigg(I_n+\frac{tA}{l}\bigg)_{i_0j}=1+\frac{tA_{i_0i_0}}{l}+\sum_{j\neq i_0}\frac{tA_{i_0j}}{l}=1-\frac{t}{l}\sum_{k\in N(i_0)}\omega_{i_0k}+\frac{t}{l}\sum_{j\in N(i_0)}\omega_{i_0j}=1.
    \end{equation*}
    Thus we also have \[\sum_{j=1}^n\bigg(I_n+\frac{tA}{l}\bigg)_{ij_0}=1\] so $\bigg(I_n+\cfrac{tA}{l}\bigg)$ is a transition probability matrix.

    Let $e=(1, ..., 1)$. Then $e^T\bigg(I_n+\cfrac{tA}{l}\bigg)^l=e^T, \bigg(I_n+\cfrac{tA}{l}\bigg)e=e$.
    By sending $l\to\infty$, we have $e^Te^{tA}=e^T, e^{tA}e=e^{tA}$. Further $\bigg(I_n+\cfrac{tA}{l}\bigg)^l\geq 0$ when $l>>1$, so $(e^{tA})_{ij}\geq 0$ for all $(i,j)\in \{1, ..., n\}\times \{1, ..., n\}$. 
\end{proof}

\begin{lemma}\label{lem:3.2}
    Assume $P$ is a transition probability matrix. If $\mu\in\P(G), P\mu\in\P(G)$. 
\end{lemma}
\begin{proof}
    $(P\mu)_i=\sum P_{ij}\mu_j\geq 0$\\
    $\sum_{i=1}^n(P\mu_i)=\sum_{i,j=1}^{n}P_{ij}\mu_j=\sum_{j=1}^n\mu_j\sum_{i=1}^nP_{ij}=\sum_{j=1}^n\mu_j=1$
\end{proof}
\begin{lemma}
    $\div_\mu(\nabla_G\log_\mu)=A\mu\quad \forall \mu\in\P_0(G)$.
\end{lemma}
\begin{proof}
    \begin{align*}
        \langle \div_\mu(\nabla_G\log\mu)\rangle_i&=\sum_{j\in N(i)}\sqrt{\omega_{ij}}g_{ij}(\mu)(\nabla_G\log\mu)_{ji} \\
        &=\sum_{j\in N(i), \mu_i\mu_j\neq 0}\omega_{ij}\frac{\mu_i-\mu_j}{\log\mu_i\log\mu/j}(\log\mu_j-\log\mu_i) \\
        &=\sum_{j\in N(i)}\omega_{ij}(\mu_j-\mu_i)
    \end{align*}
    \begin{align*}
        (A\mu)_i=\sum_{j=1}^nA_{ij}\mu_j=\sum_{j\in N(i)}A_{ij}\mu_j+A_{ii}\mu_i=\sum_{j\in N(i)}\omega_{ij}(\mu_j-\mu_i)
    \end{align*}
\end{proof}

\begin{remark}\label{rem: 3.4}
    Set $B(t)=e^{tA}$. Then $\dot{B}(t)=Ae^{tA}$.
\end{remark}

\begin{theorem}
    Let $u_0:\P(G)\to\R$ be Fr\'echet continuously differentiable.
    \begin{itemize}
        \item [(i)] For every $\mu\in\P(G)$, 
    \begin{equation}
        \begin{cases}
        \dot\sigma^\mu=\div_{\sigma^\mu}(\nabla_G\log\sigma^\mu)=A\sigma^\mu \\
        \sigma^\mu(0)=\mu
        \end{cases}     
    \end{equation}
    admits a unique solution $\sigma^\mu:[0,T]\to\P(G)$.\\
    \item [(ii)] $u(t,\mu)=u_0(\sigma^\mu)$ satisfies \eqref{HJE}.
    \end{itemize}
\end{theorem}
\begin{proof}
    Let $\sigma^\mu(t)=e^{tA}\mu$. From Lemma \ref{lem:3.1}
 and Lemma \ref{lem:3.2}, $\sigma^\mu: [0,T]\to\P(G)$. Using Remark \ref{rem: 3.4}, $\dot\sigma^\mu(t)=Ae^{tA}\mu=A\sigma^\mu$. The solution is unique since $\sigma^\mu\mapsto A\sigma^\mu$ is $\Vert A\Vert$-lipschitz.

Let $f\in\R^n$ such that $f=\bar{\mu}-\mu$ for some $\bar{\mu}\in\P(G)$. Note that $\sum f_i=0$, so $\sum (e^{tA}f)_i=0$.
\begin{align*}
    &\lim_{\epsilon\to 0}\frac{u(t, \mu+\epsilon f)-u(t,\mu)}{\epsilon}\\
    &=\lim_{\epsilon\to 0}\frac{u_0(e^{tA}\mu+\epsilon e^{tA}f)-u_0(e^{tA}\mu)}{\epsilon}\\
    &=\bigg(\frac{\delta u_0}{\delta\mu}(e^{tA}\mu), e^{tA}f\bigg)\\
    &=\bigg(e^{tA}\frac{\delta u_0}{\delta\mu}(e^{tA}\mu), f\bigg)
\end{align*}
From Theorem \ref{thm:2.14} $\nabla_\W u(t,\mu)=\nabla_G(\frac{\delta u}{\delta\mu}(t,\mu))=\nabla_G(e^{tA}\frac{\delta u_0}{\delta\mu}(e^{tA}\mu))$.
\begin{align*}
    \Delta_{\text{ind}}u(t,\mu)&=-(\nabla_\W u(t,\mu),\nabla_G\log_\mu)_\mu \\
    &=-\bigg(\nabla_G\bigg(e^{tA}\frac{\delta u_0}{\delta\mu}(e^{tA}\mu)\bigg),\nabla_G\log\mu\bigg)_\mu \\
    &=\bigg(\frac{\delta u_0}{\delta\mu}(e^{tA}\mu), e^{tA}\div_\mu(\nabla_G\log\mu)\bigg)
\end{align*}

If $|e^{tA}\mu-e^{t_0A}\mu|$ is sufficiently small, 

\begin{align*}
    &\bigg|\frac{u_0(e^{tA}\mu)-u_0(e^{t_0A}\mu)}{t-t_{0}}+(\nabla_\W u_0(e^{t_0A}\mu)-\nabla_G\log (e^{t_0A}\mu))_{e^{t_0A}\mu}\bigg| \\
    &\leq \frac{\epsilon\W(e^{tA}\mu, e^{t_0A}\mu)}{|t-t_0|}+c\bigg\Vert \frac{e^{tA}\mu-e^{t_0A}\mu}{t-t_0}-\div_{e^{tA}\mu}(\nabla_G\log (e^{t_0A}\mu))\bigg\Vert \\
    &\leq \epsilon\int_{t_0}^{t}\frac{\Vert\nabla_G(\log (e^{tA}\mu(s)))\Vert_{e^{tA}\mu(s)}^2}{t-t_0}+c\bigg\Vert \frac{e^{tA}\mu-e^{t_0A}\mu}{t-t_0}-\div_{e^{tA}\mu}(\nabla_G\log (e^{t_0A}\mu))\bigg\Vert \\
    &\underset{t\to t_0}{\longrightarrow} \epsilon\Vert\nabla_G(\log (e^{t_0A}\mu))\Vert_{e^{t_0A}\mu}^2
\end{align*}
Since $\epsilon>0$ is arbitrary, $\cfrac{d}{dt}u_0(e^{tA}\mu)|_{t=t_0}=-(\nabla_\W u_0(e^{t_0A}\mu),\nabla_G\log (e^{t_0A}\mu))_{e^{tA}\mu}$.

Thus 
\begin{align*}
    \frac{\partial u}{\partial t}(t,\mu)&=\frac{\partial}{\partial t}(u_0(e^{tA}\mu)) \\
    &=-(\nabla_\W u_0(e^{tA}\mu),\nabla_G\log (e^{tA}\mu))_{e^{tA}\mu} \\
    &=\bigg(\frac{\delta u_0}{\delta\mu}(e^{tA}\mu), \div_{e^{tA}\mu}\nabla_G\log(e^{tA}\mu)\bigg)
\end{align*}

Therefore we conclude that $\cfrac{\delta u}{\delta t}(t,\mu)=\Delta_{\text{ind}} u(t,\mu)$
\end{proof}

\section{Solutions to HJE}
In this section we are going to assume that 
\begin{itemize}
    \item $g(s,t)=0$ if $st=0$. \\
    \item Set $\bar{g}(s,t)=(\log s-\log t)g(s,t)$ if $s,t>0$. $\bar{g}$ admits a unique extension on $[0,+\infty)^2$ which is uniquely determined on $[0, 1]^2$. Hence $\mu\to \div_\mu(\nabla_G\log\mu)$ admits an extension $\xi$ on $[0, + \infty)^n$ which is unique on $\P(G)$. \\
    \item $\xi$ is Lipschitz of class $C^1$.
\end{itemize}

\begin{theorem}\label{thm:4.2}
    If $\mu\in\P(G)$, then 
    \begin{equation}
        \begin{cases}
        \dot\sigma=\xi(\sigma)=\div_{\sigma}(\nabla_G\log\sigma)\\
        \sigma(0)=\mu
        \end{cases}
    \end{equation}
    admits a unique solution $\sigma:[0,T]\to\P(G)$ for any $T>0$.
\end{theorem}
\begin{proof}
    If $\mu_i=0$ then $(\div_\mu(\nabla_G\log\mu))_i=0$ and so $\dot\sigma_i=(\div_\sigma(\nabla_G\log\sigma))_i$ is satisfied by $\sigma_i(t)=\mu_i \quad\forall t\in [0, T]$. 

    Let $I$ be the set of $i$ such that $\mu_i=0$ and $j\in I^c$. Then
    \begin{equation*}
        (\div_\mu(\nabla_G\log\sigma))_j=\sum_{k\in N(j)}\omega_{jk}\bar{g}(\sigma,\sigma_j)=\sum_{k\in N(j)\setminus I}\omega_{jk}\bar{g}(\sigma_k,\sigma_j)
    \end{equation*} 
    Thus we can reduce the problem from $i\in \{1, ..., n\}$ to $i\in I^c$.
    
    For simplicity, assume $I^c=\{1, 2, ..., k\}$ and $I=\{k+1, ..., n\}$.
    
    Define $\tilde\sigma=\begin{pmatrix}
    \sigma_1 \\
    \vdots \\
    \sigma_k
    \end{pmatrix}$ and $\tilde\mu=\begin{pmatrix}
    \mu_1 \\
    \vdots \\
    \mu_k
    \end{pmatrix}$.
    So we want to solve
    \begin{equation}\label{reduced_cont}
        \begin{cases}
            \dot{\tilde\sigma}=\xi(\tilde\sigma) \\
            \tilde\sigma(0)=\tilde\mu
        \end{cases}
    \end{equation}
    By Carath\'eodory's existence theorem there exists $\delta>0$ such that \eqref{reduced_cont} admits a solution $\tilde\sigma:[0,\delta)\to \R^k$. Since $\tilde\sigma(0)\in (0,1)^k$ and $\tilde\sigma$ is continuous, there exists $\bar\delta\leq\delta$ such that $\tilde\sigma:[0,\tilde\sigma]\to[0,1]^k$. Further there exists a largest $T_1$ such that $\tilde\sigma([0,T_1])\subset [0,1]^k$.Then $\sum_{j=1}^k\tilde\sigma_j(0)=1$ and $\sum_{j=1}^l\dot{\tilde\sigma}_j=\sum_{j=1}^k\sum_{i\in N(j)}\omega_{ij}\bar{g}(\mu_j,\mu_i)=0$, and so $\sum_{j=1}^k\tilde\sigma_j=1$ for all $t\in [0, T_1]$. 

    Assume $\tilde\sigma_1(T_1), ..., \tilde\sigma_m(T_1)>0, \tilde\sigma_{m+1}(T_1)=\cdots=\tilde\sigma_k(T_1)=0$.

    If $m=k$ then $T_1=T$.

    If $T_1<T$, repeat the procedure with $(\sigma_1, ..., \sigma_m)$ on $[T_1, T_2]$.

    If $T_2=T$ we are done.

    Since there are only finitely many steps, eventually we achieve a solution on $[0,T]$.

    Since $\xi$ is Lipschitz, the solution is unique.
\end{proof}

We will now denote the solution from \ref{thm:4.2} as $\sigma^\mu$ and define $u(t, \mu)=u_0(\sigma^\mu(t))$.
\begin{theorem}\label{thm:sigma_C1}
    Suppose $F\in C^1(\R^n, \R^n)\cap\text{Lip}(\R^n,\R^n)$. Let $\sigma:[0,T]\times\R^n\to\R^n$ be the unique solution to 
    \begin{equation}\label{eq:continuity}
        \begin{cases}
            \dot\sigma(t,\mu)=F(\sigma(t,\mu)) \\
            \sigma(0,\mu)=\mu
        \end{cases}
    \end{equation}
    Then $\sigma$ is of class $C^1$ on $(0,T)\times \R^n$. 
\end{theorem}
\begin{proof}
    If we can show that $\sigma$ is continuous then $F\circ \sigma$ will be continuous and by \eqref{eq:continuity}, $\sigma(\cdot,\mu)$ would be continuously differentiable. It remains to show that $\sigma(t,\cdot)$ is continuously differentiable. 
    Let $\xi\in\R^n$. 
    \begin{align*}
        \sigma(t, \mu+\epsilon)-\sigma(t,\mu)|&=|\xi+\int_0^t F(\sigma(s,\mu+\epsilon))-F(\sigma(s,\mu))ds| \\
        &\leq |\xi|+\int_0^t\text{lip}F|\sigma(s,\mu+\epsilon)-\sigma(s,\mu)|ds \\
        &\leq |\xi|e^{t\text{lip}F}
    \end{align*}
    The last inequality is true from Gronwall's inequality.

    Thus $\sigma(t, \cdot)$ is Lipschitz and $\text{Lip}\sigma(t, \cdot)\leq e^{t\text{lip}F}$ and so $t\mapsto\sigma(t,\mu)$ is of class $C^1$.

    Let $e=(0, ..., 1, 0, ..., 0)$ where $1$ is on the $j-$th term.

    \[\sigma_i(t, \mu+\epsilon e)=\mu_i+\epsilon e_i+\int_0^t F_i(\sigma(s, \mu+\epsilon e_i))ds\]
    Set $\psi_\epsilon(t)=\cfrac{\sigma(t, \mu+\epsilon e)-\sigma(t, \mu)}{\epsilon}$.
    \begin{equation}\label{eq:0}
        \cdot\psi_{\epsilon}(t)=\frac{F(\sigma(t, \mu+\epsilon e))-F(\sigma(t,\mu))}{\epsilon}
    \end{equation}
    Note that
    \begin{equation}
        F(b)-F(a)=\int_0^1\frac{d}{dt}F(a+t(b-a))dt=\bigg(\int_0^1\nabla F(a+t(b-1))dt\bigg)(b-a)=:G(a,b)(b-a)
    \end{equation}
    Then $G$ is continuous, and $G(a,a)=\nabla F(a)$.
    By \eqref{eq:0} 
    \begin{equation}\label{eq:3}
        \dot\psi_\epsilon=G(\sigma(t,\mu),\sigma(t,\mu+\epsilon e))(\sigma(t,\mu+\epsilon e)-\sigma(t,\mu)).
    \end{equation}
    Since $\sigma$ is continuous, it is bounded on the compact set $[0,1]\times B_1(\mu)$.

    Use \eqref{eq:3} to conclude 
    \begin{equation}\label{eq:5}
        \sup_{\substack{t\in [0,1]\\ \epsilon\in [-1,1]}}|\dot\psi_\epsilon(t)|<+\infty
    \end{equation}

    Also \begin{align}\label{eq:6}
        |\psi_\epsilon(t)|=\bigg|\frac{\sigma(t,\mu+\epsilon e)-\sigma(t,\mu)}{\epsilon}\bigg|\leq \frac{|\epsilon e|e^{t\text{lip}F}}{\epsilon}\leq e^{t\text{lip}F}
    \end{align}
    From \eqref{eq:5} and \eqref{eq:6} we can apply Ascoli-Arzella theorem to conclude that for all $(\epsilon_k)_k\to 0$, there exists a subsequence $(k_l)_l$ such that $(\psi_{\epsilon_{k_l}})_l\to \psi$ for some $\psi\in C([0,T], \R^n)$.

    Using \eqref{eq:3}, 
    \begin{align*}
        \psi(t)=\lim_{l\to 0}\psi_{\epsilon_{k_l}}(t)&=\lim_{l\to\infty} e+\int_0^t G(\sigma(t,\mu),\sigma(t,\mu+\epsilon_{k_l} e))(\sigma(t,\mu+\epsilon_{k_l}e-\sigma(t,\mu))dt \\
        &=e+\int_0^t\nabla F(\sigma(t,\mu))\psi(t)dt
    \end{align*}

    Thus $\psi$ is of class $C^1$, and 
    \[
    \begin{cases}
        \dot\psi = \nabla F(\sigma(t,\mu))\psi(t) \\
        \psi(0)=e
    \end{cases}
    \]
    admits a unique solution, and so $\lim_{l\to\infty}\psi_{\epsilon_{k_l}}$ is independent of the chosen subsequence. Hence we conclude $\lim_{\epsilon\to 0}\psi_\epsilon(t)=\psi(t)$ exists. This shows that $\cfrac{\delta\sigma}{\delta e}(t,\mu)$ exists.

    Define \[H(t,\mu,M)=\nabla F(\sigma(t,\mu))M\quad t\in[0,T], \mu\in\P(G), M\in\R^{n\times n}\] Note that $H$ is continuous.
    We plan to show that $\psi(t, \cdot)$ is continuous. Let $(\mu_n)_n\subset\P(G)$ be a sequence that converges to $\mu$. We want to show that $\psi(t,\mu_n)\to\psi(t, \mu)$.

    Suppose $|\nabla F|<\leq c$ for some $c$. For $\epsilon>0$,
    \[\frac{d}{dt}\sqrt{\epsilon+|\psi|^2}=\frac{1}{2}\frac{2\langle\psi,\dot\psi\rangle}{\sqrt{\epsilon+|\psi|^2}}=\frac{\langle\psi,\nabla F(\sigma)\psi\rangle}{\sqrt{\epsilon+|\psi|^2}}\leq \frac{c|\psi|^2}{\sqrt{\epsilon+|\psi|^2}}\leq c\sqrt{\epsilon+|\psi|^2}\]

    So by Gronwall's inequality, \[\sqrt{\epsilon+|\psi(t,\mu)|^2}\leq\sqrt{\epsilon+|\psi(0,\mu)|^2}e^{ct}\]
    Let $\epsilon\to 0$, and obtain 
    \[|\psi(t,\mu)\leq|\psi(0,\mu)|e^{ct}=e^{ct}\] This shows that $(\psi(t,\mu_n))_n$ is bounded.
    Further \[|\dot\psi(t,\mu_n)|\leq c|\psi(t,\mu_n)|\leq ce^{ct}\] This shows that $(\psi(t,\mu_n))_n$ is equicontinuous.

    From Ascoli-Arzella theorem, $(\psi(t,\mu_n))_n$ has a subsequence $(n_l)_l$ which converges uniformly on $[0,T]$ to some $f$.

    Thus \[f=\lim_{l\to\infty}\psi(t,\mu_{n_l})=\lim_{l\to\infty}e+\int_0^t\nabla F(\sigma(s,\mu_{n_l}))\psi(s,\mu_{n_l})ds=e+\int_0^t\nabla F(\sigma(s,\mu))f(s)ds\]

    Thus we have \[\begin{cases}
        \dot{f}=\nabla F(\sigma(t,\mu))f(t) \\
        f(0)=e
    \end{cases}\]
    Since this has a unique solution $f(t)=\psi(t,\mu)$, $\psi$ is continuous.
   \end{proof}
\begin{lemma}\label{lem:chain-rule}
    Assume $A:\P(G)\to\P(G)$ is continuously Fr\'echet differentiable at $\mu_0$. Assume $v:\P(G)\to\R$ is continuously Fr\'echet differentiable at $\nu_0=A(\mu_0)$. Then $v\circ A$ is differentiable at $\mu_0$ and $\frac{d}{dt}v(A(\mu_0+tf))|_{t=t_0}=\bigg((\nabla_\mu A(\mu_0))^T\frac{\delta v}{\delta\mu}(A(\mu_0), f\bigg)$
\end{lemma}
\begin{proof}
Let $\nu\in\P(G)$ and $f=\nu-\mu_0$. Set $\sigma_t=A(\mu_0+tf)=A(\mu_0)+\nabla\mu A(\mu_0)ft+o(t)$ and denote $\nabla\mu A(\mu_0)f$ as $g$. Then $\sum g_i=\sum \dot\sigma_i=0$.  

Then \begin{align*}
    v(A(\mu_0+tf))&=v(A(\mu_0)+tg+o(t)) \\
    &v(A(\mu_0))+\bigg(\frac{\delta v}{\delta\mu}(A(\mu_0)), tg+o(t)\bigg)+o(tg+o(t)) \\
    &=v(A(\mu_0))+t\bigg(\frac{\delta v}{\delta\mu}(A(\mu_0)),g\bigg)+o(t) \\
    &= v(A(\mu_0))+t\bigg((\nabla_\mu A(\mu_0))^T\frac{\delta v}{\delta\mu}(A(\mu_0))\bigg)+o(t)
\end{align*}
\end{proof}
\begin{corollary}\label{cor:cont-diff}
    $v\circ A$ is differentiable in the sense of Wasserstein.
\end{corollary}
\begin{proof}
    $\mu\mapsto(\nabla_{\bar\mu} A(\mu))^T$ is continuous, and $\mu\mapsto\frac{\delta v}{\delta\mu}\circ A$ is continuous. Hence $\mu\mapsto \frac{\delta}{\delta\mu}(v\circ A)(\mu)$ is continuous. Apply Theorem \ref{thm:2.14} to conclude that $v\circ A$ has a $\W$-gradient at $\mu$.
\end{proof}
\begin{theorem}
    \begin{itemize}
        \item [(i)] Set $u(t,\mu)=u_0(\sigma(t,\mu))$ where $u_0:\P(G)\to\R$ is Fr\'echet continuously differentiable. Then $u$ is continuously differentiable.
        \item [(ii)] $\partial_tu
        =\Delta_{\text{ind}}u(t,\mu)$
    \end{itemize}
\end{theorem}
\begin{proof}
    \begin{itemize}
        \item[(i)] We use Corollary \ref{cor:cont-diff} in $v=u_0$ to conclude that $u(t,\cdot)$ is continuously differentiable. Further since $\sigma(\cdot, \mu)$ is differentiable and $u_0$ is continuously Fr\'echet differentiable, apply Lemma \ref{lem:chain-rule} to conclude that $u(\cdot,\mu)$ is differentiable.
        \item[(ii)] Define $r(s)=\sigma(s,\sigma(h,\mu))$. In other words \[\begin{cases}
            \dot{r}=\xi(r) \\
            r(0)=\sigma^\mu(h)
        \end{cases}\] on $(0, t-h)$.\\
        We have \begin{equation}\label{eq:u-shift}
            u(t-h, \sigma^\mu(h))=u_0(\sigma^\mu(t))=u(t,\mu)
        \end{equation}
        
        \begin{align*}
        u(t-h,\sigma^\mu(h))&=u(t)-h\partial_t u(t, \sigma(0))+h(\nabla_\W u(t,\sigma(0)),-\nabla_G\log\sigma(0))_{\sigma(0)}+o(h) \\
        &=u(t,\mu)-h\partial_t u(t,\mu)-h(\nabla_\W u(t,\mu),\nabla_\W\log\mu)_\mu+o(h) \\
        \end{align*}

        From \eqref{eq:u-shift}, 
        \[-\partial_t u(t,\mu)-(\nabla_\W u(t,\mu),\nabla_G\log\mu)_\mu+\frac{o(h)}{h}=-\partial_t u(t,\mu)+(\div_\mu(\nabla_\W u(t,\mu)),\log\mu)=0\]
    \end{itemize}
\end{proof}

\section*{Acknowledgements}
    The author would like to thank Professor Wilfrid Gangbo and Mohit Bansil for their mentorship.

\end{document}